\documentclass[12pt]{article}

\usepackage{geometry}   % page dimensions are set in oldstyle.tex

\usepackage[dvipsnames,table]{xcolor}
\definecolor{c0}{HTML}{ffb000}
\definecolor{c1}{HTML}{fe6100}
\definecolor{c2}{HTML}{dc267f}
\definecolor{c3}{HTML}{785ef0}
\definecolor{c4}{HTML}{648fff}

\usepackage{amsmath,amssymb,amsthm,mathtools}
\usepackage{graphicx}
\usepackage{tikz}
\usetikzlibrary{arrows.meta}
\usepackage[labelfont=bf]{caption}
\usepackage{enumitem}
\usepackage{algorithm}
\usepackage{algpseudocode}

\usepackage{microtype}
\usepackage{hyperref}

\usepackage[capitalize,nameinlink,noabbrev]{cleveref}
\crefformat{equation}{#2(#1)#3}
\crefmultiformat{equation}{#2(#1)#3}%
{ and~#2(#1)#3}{, #2(#1)#3}{ and~#2(#1)#3}
\crefrangeformat{equation}{#3(#1)#4--#5(#2)#6}
\crefname{algorithm}{Algorithm}{Algorithms}

\newtheorem{theorem}{Theorem}[section]
\newtheorem{lemma}[theorem]{Lemma}
\newtheorem{proposition}[theorem]{Proposition}
\newtheorem{corollary}[theorem]{Corollary}
\newtheorem{assumption}[theorem]{Assumption}
\theoremstyle{definition}
\newtheorem{definition}[theorem]{Definition}
\theoremstyle{remark}
\newtheorem{remark}[theorem]{Remark}

\usepackage{bm}
\newcommand{\R}{\mathbb R}
\newcommand{\norm}[1]{\lVert#1\rVert}
\newcommand{\abs}[1]{\lvert#1\rvert}
\newcommand{\fl}{\operatorname{fl}}
\newcommand{\diag}{\operatorname{diag}}

\newcommand{\spec}{\operatorname{spec}}

\newcommand{\col}{\operatorname{col}}
\newcommand{\eps}{\varepsilon}

\newcommand{\T}{\mathsf{T}}
\renewcommand{\vec}[1]{#1}

\newcommand{\Vk}{\vec V_k}
\newcommand{\Vkp}{\vec V_{k+1}}
\newcommand{\vkp}{\vec v_{k+1}}
\newcommand{\Tk}{\vec T_k}
\newcommand{\uTk}{\ul{\vec T}_k}
\newcommand{\Fk}{\vec F_k}
\newcommand{\Rk}{\vec R_k}
\newcommand{\Rkp}{\vec R_{k+1}}
\newcommand{\Mk}{\vec M_k}

\newcommand{\Wk}{\vec W_k}
\newcommand{\Wkp}{\vec W_{k+1}}
\newcommand{\Yk}{\vec Y_k}
\newcommand{\Ykp}{\vec Y_{k+1}}
\newcommand{\Nk}{\vec N_k}

\newcommand{\bVk}{\bar{\vec V}_k}
\newcommand{\bVkp}{\bar{\vec V}_{k+1}}
\newcommand{\bvkp}{\bar{\vec v}_{k+1}}
\newcommand{\bFk}{\bar{\vec F}_k}
\newcommand{\bRk}{\bar{\vec R}_k}
\newcommand{\bRkp}{\bar{\vec R}_{k+1}}
\newcommand{\bGkp}{\bar{\vec G}_{k+1}}
\newcommand{\bMk}{\bar{\vec M}_k}

\newcommand{\Ck}{\vec C_k}
\newcommand{\Ckp}{\vec C_{k+1}}
\newcommand{\Kk}{\vec K_k}
\newcommand{\Kkp}{\vec K_{k+1}}
\newcommand{\Lk}{\vec L_k}

\newcommand{\hA}{\widehat{\vec A}}
\newcommand{\hE}{\widehat{\vec E}}
\newcommand{\hVk}{\widehat{\vec V}_k}
\newcommand{\hVkp}{\widehat{\vec V}_{k+1}}
\newcommand{\hvkp}{\widehat{\vec v}_{k+1}}

\newcommand{\tA}{\widetilde{\vec A}}

\newcommand{\tv}{\widetilde{\vec v}}

\newcommand{\ul}[1]{\mkern1.5mu\underline{\mkern-1.5mu\smash[b]{#1}\mkern-1.5mu}\mkern1.5mu}

\usepackage[backend=biber,backref,style=alphabetic,sorting=nyt,maxcitenames=99,maxbibnames=99,giveninits=true,url=false,doi=false,isbn=false,eprint=false,date=year]{biblatex}
\DeclareSourcemap{
  \maps[datatype=bibtex, overwrite]{
    \map{
      \step[fieldset=editor, null]
    }
  }
}
\DefineBibliographyStrings{english}{%
  backrefpage = {cited on page},%
  backrefpages = {cited on pages}%
}

\usepackage{etoolbox}

\usepackage[T1]{fontenc}
\usepackage[helvratio=0.94]{newtxtext}
\usepackage[bigdelims,vvarbb]{newtxmath}
\usepackage[cal=boondoxo]{mathalfa}
\newcommand{\headfont}{\sffamily\bfseries}

\renewcommand{\vec}[1]{\mathbf{#1}}

\providecommand{\textls}[2][]{#2}
\newcommand{\lsp}[1]{\textls[90]{#1}}

\usepackage[export]{adjustbox}
\setkeys{Gin}{max width=0.95\linewidth}

\selectfont
\usepackage{fancyhdr}
\newcommand{\laarunningtitle}{}
\newcommand{\laarunningauthor}{}
\newcommand{\setrunningheads}[2]{%
  \renewcommand{\laarunningtitle}{#1}%
  \renewcommand{\laarunningauthor}{#2}%
}
\newcommand{\laahead}[1]{\sffamily\fontsize{8}{10}\selectfont\lsp{\MakeUppercase{#1}}}
\fancypagestyle{laa}{%
  \fancyhf{}%
  \fancyhead[L]{%
    \ifodd\value{page}\else
      {\headfont\fontsize{8}{10}\selectfont\thepage}\hspace{1.6em}%
      {\laahead{\laarunningauthor}}%
    \fi}%
  \fancyhead[R]{%
    \ifodd\value{page}%
      {\laahead{\laarunningtitle}}\hspace{1.6em}%
      {\headfont\fontsize{8}{10}\selectfont\thepage}%
    \fi}%
}
\fancypagestyle{laafirst}{%
  \fancyhf{}%
  \fancyfoot[L]{%
    \sffamily\fontsize{8}{10}\selectfont\laajournalline}%
  \fancyfoot[R]{\headfont\fontsize{8}{10}\selectfont\thepage}%
}
\newcommand{\laajournalline}{}

\newcommand{\laasubmittedby}{}

\newcommand{\laaaffiliation}{}
\newcommand{\affiliation}[1]{\renewcommand{\laaaffiliation}{#1}}

\makeatletter
\renewcommand{\maketitle}{%
  \thispagestyle{laafirst}%
  \begingroup
  \raggedright
  {\headfont\fontsize{20}{22}\selectfont\@title\par}
  \vskip 1.8em
  {\sffamily\fontsize{11}{14}\selectfont\@author\par}
  \ifdefempty{\laaaffiliation}{}{%
    \vskip 0.3em
    {\itshape\fontsize{9.5}{12.5}\selectfont\laaaffiliation\par}}
  \ifdefempty{\laasubmittedby}{}{%
    \vskip 1.2em
    {\sffamily\fontsize{8.5}{11}\selectfont Submitted by \laasubmittedby\par}}
  \endgroup
  \vskip 1.4em
}

\renewenvironment{abstract}{%
  \vskip 0.5em
  \noindent\rule{\textwidth}{0.4pt}
  \vskip 1.1em
  \noindent{\headfont\fontsize{8}{10}\selectfont\lsp{ABSTRACT}}\par
  \vskip 0.8em
  \small\noindent\ignorespaces
}{%
  \par\vskip 1.1em
  \noindent\rule{\textwidth}{0.4pt}
  \vskip 1.6em
}
\makeatother

\usepackage{titlesec}
\titleformat{\section}[block]
  {\normalfont\bfseries}
  {\makebox[2.1em][l]{\thesection.}}
  {0pt}
  {}
\titlespacing*{\section}{0pt}{2.0\baselineskip}{0.75\baselineskip}

\titleformat{\subsection}[block]
  {\normalfont\itshape}
  {\makebox[2.8em][l]{\thesubsection.}}
  {0pt}
  {}
\titlespacing*{\subsection}{0pt}{1.45\baselineskip}{0.5\baselineskip}

\apptocmd{\appendix}{%
  \titleformat{\section}[block]
    {\normalfont\bfseries}
    {\appendixname\ \thesection.}{0.5em}{}%
  \titlespacing*{\section}{0pt}{2.0\baselineskip}{0.75\baselineskip}%
}{}{}

\makeatletter
\newtheoremstyle{laaplain}
  {1.1\baselineskip plus 4pt minus 2pt}%     % space above
  {1.1\baselineskip plus 4pt minus 2pt}%     % space below
  {\itshape}                             % body font
  {0pt}                                  % indent
  {\normalfont\bfseries}                 % head font
  {.}                                    % punctuation after head
  {0.6em}                                % space after head
  {\thmname{#1}\thmnumber{ #2}\thmnote{ \normalfont(#3)}}
\newtheoremstyle{laaremark}
  {1.1\baselineskip plus 4pt minus 2pt}%
  {1.1\baselineskip plus 4pt minus 2pt}%
  {\normalfont}
  {0pt}
  {\normalfont\bfseries}
  {.}
  {0.6em}
  {\thmname{#1}\thmnumber{ #2}\thmnote{ \normalfont(#3)}}
\makeatother

\theoremstyle{laaplain}
\newtheorem{laatheorem}{Theorem}[section]
\newtheorem{laalemma}[laatheorem]{Lemma}
\newtheorem{laaproposition}[laatheorem]{Proposition}
\newtheorem{laacorollary}[laatheorem]{Corollary}
\newtheorem{laaassumption}[laatheorem]{Assumption}
\theoremstyle{laaremark}
\newtheorem{laadefinition}[laatheorem]{Definition}
\newtheorem{laaremark}[laatheorem]{Remark}

\makeatletter
\newcommand{\laakeeptogether}[1]{%
  \AtBeginEnvironment{#1}{%
    \interlinepenalty=9999 \postdisplaypenalty=9999
    \@itempenalty=9999 }}
\makeatother
\laakeeptogether{theorem}     \laakeeptogether{lemma}
\laakeeptogether{proposition} \laakeeptogether{corollary}
\laakeeptogether{assumption}  \laakeeptogether{definition}
\laakeeptogether{remark}

\makeatletter
\renewenvironment{proof}[1][\proofname]{%
  \par\pushQED{\qed}%
  \normalfont\topsep6\p@\@plus6\p@\relax
  \trivlist\item[\hskip\labelsep\itshape#1.]\ignorespaces
}{%
  \popQED\endtrivlist\@endpefalse
}
\makeatother
\renewcommand{\proofname}{Proof}

\numberwithin{equation}{section}

\makeatletter
\g@addto@macro\normalsize{%
  \setlength{\abovedisplayskip}{9pt plus 6pt minus 3pt}%
  \setlength{\belowdisplayskip}{9pt plus 6pt minus 3pt}%
  \setlength{\abovedisplayshortskip}{4pt plus 2pt}%
  \setlength{\belowdisplayshortskip}{5pt plus 2pt minus 2pt}%
}
\makeatother

\crefname{laatheorem}{Theorem}{Theorems}
\Crefname{laatheorem}{Theorem}{Theorems}
\crefname{laalemma}{Lemma}{Lemmas}
\Crefname{laalemma}{Lemma}{Lemmas}
\crefname{laacorollary}{Corollary}{Corollaries}
\Crefname{laacorollary}{Corollary}{Corollaries}
\crefname{laaproposition}{Proposition}{Propositions}
\Crefname{laaproposition}{Proposition}{Propositions}
\crefname{laaassumption}{Assumption}{Assumptions}
\Crefname{laaassumption}{Assumption}{Assumptions}
\crefname{laadefinition}{Definition}{Definitions}
\Crefname{laadefinition}{Definition}{Definitions}
\crefname{laaremark}{Remark}{Remarks}
\Crefname{laaremark}{Remark}{Remarks}

\hypersetup{colorlinks=true, linkcolor=black, citecolor=black,
            filecolor=black, urlcolor=black}

\defbibheading{bibliography}[References]{%
  \section*{#1}}

\setrunningheads{Lanczos in finite precision}{T. Chen}
\affiliation{New York University Shanghai\\567 West Yangsi Road\\Shanghai, 200126, P.R. China}

\begin{document}

% Title, abstract, and body of the paper.  It contains no preamble and no
% \begin{document}; the wrapper finite_precision_lanczos_oldstyle.tex
% supplies those.

\title{A simple stability analysis of the Lanczos algorithm in finite precision arithmetic}
\author{Tyler Chen}
\date{}

\maketitle

\begin{abstract}
We give a self-contained analysis of the symmetric Lanczos algorithm without reorthogonalization when computations are performed inexactly.
We show the computed Lanczos vectors are well-behaved locally.
From this we (i) prove Paige's loss-of-orthogonality identity, and (ii) construct a Greenbaum-type backward stability result, exhibiting a nearby problem on which exact Lanczos produces the computed tridiagonal matrix. This allows us to derive several facts about the computed Ritz values at the original eigenvalues.
Our analysis simplifies the analyses of Paige and Greenbaum, at the cost of hiding polynomial factors in the iteration count.
\end{abstract}

\section{Introduction}

Let \(\vec A=\vec A^\T\in\R^{n\times n}\) and let $\vec{v}\in\R^n$ be nonzero.
In exact arithmetic, \(k\) non-terminating steps of the Lanczos algorithm (\cref{alg:lanczos}) applied to \((\vec A,\vec v)\) produce a matrix \(\Vk = [\vec{v}_1, \ldots, \vec{v}_k]\) with orthonormal columns, \(\vec v_1=\vec v/\norm{\vec v}\), and a symmetric tridiagonal matrix \(\Tk\in\R^{k\times k}\) satisfying
\[
 \vec A\Vk=\Vk\Tk+\beta_k\vkp\vec e_k^\T.
\]
Hence, for $j\leq k+1$, the vectors $\{\vec{v}_1, \ldots, \vec{v}_{j}\}$ form an orthonormal basis for the Krylov subspace
\[
\operatorname{span}\{\vec{v},\vec{A}\vec{v}, \ldots, \vec{A}^{j-1}\vec{v}\}.
\]
The matrices $\vec{V}_k$ and $\vec{T}_k$ can then be used for downstream tasks such as estimating eigenvalues, solving linear systems of equations, and problems involving matrix functions \cite{meurant2006,greenbaum1997,saad2011,tropp_webber2023,chen2024}.

It was known since the introduction of the Lanczos algorithm that the behavior of the algorithm in finite precision arithmetic can be drastically different \cite{lanczos1950}.
Indeed, the columns of \(\Vk\) need not remain mutually orthogonal (and in fact can eventually become nearly linearly dependent), and the computed tridiagonal matrix may differ greatly from what would have been produced in exact arithmetic.
\Cref{fig:backward-experiment} illustrates this on a small example; the same run is revisited in \cref{sec:numerical}, where the backward model is built for it explicitly.

\begin{figure}[t]
\centering
\includegraphics[scale=.86]{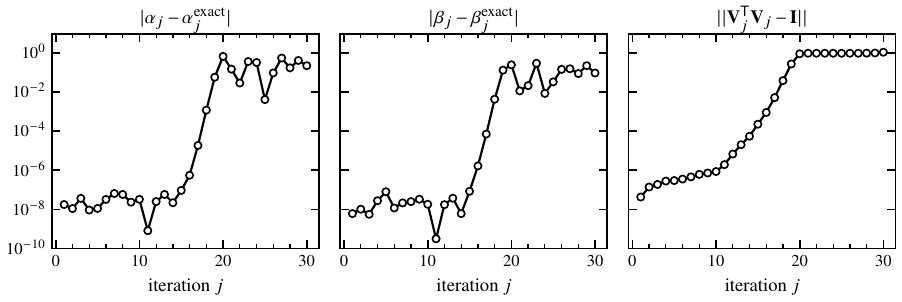}
\caption{\Cref{alg:lanczos} run for \(k=30\) iterations on a \(64\times64\) Strako\v{s} matrix \(\vec A\), in IEEE single precision, for which the unit roundoff is \(\approx10^{-7}\); see \cref{sec:numerical} for details.
Left and center: stepwise deviation of the computed coefficients \(\alpha_j\) and \(\beta_j\) from the coefficients \(\alpha_j^{\textup{exact}}\) and \(\beta_j^{\textup{exact}}\) produced by exact Lanczos (double precision with full reorthogonalization) applied to \(\vec A\).
Right: loss of orthogonality of the computed Lanczos vectors.
The deviations grow in lockstep with the loss of orthogonality: once orthogonality is gone, the computed run no longer tracks exact Lanczos on \(\vec A\), and the deviations are many orders of magnitude larger than the single-precision unit roundoff.}
\label{fig:backward-experiment}
\end{figure}

Remarkably, this instability does not render the algorithm useless.
Lanczos without reorthogonalization is used widely and successfully: Ritz values continue to converge to eigenvalues of \(\vec A\) (although possibly with spurious copies), and approximations built from the computed \(\Tk\), such as those to quadratic forms and matrix functions, are often nearly as accurate as their exact-arithmetic counterparts.
Explaining this apparent paradox has occupied numerical analysts for decades.
Much of our current understanding of the algorithm's behavior in finite precision arithmetic is due to Chris Paige and Anne Greenbaum.

Paige  showed that the algorithm is well-behaved locally, even in the presence of rounding errors \cite{paige1971,paige1972,paige1976,paige1980}.
In particular, he proved that the computed Lanczos vectors nearly satisfy a symmetric three-term recurrence, are nearly unit length, and that consecutive vectors are nearly orthogonal.
From this, he proved that loss of orthogonality of the Lanczos vectors goes hand-in-hand with the convergence of a Ritz value (eigenvalue of $\vec{T}_k$) to an eigenvalue of $\vec{A}$ \cite{paige1980}; this result is often referred to as ``Paige's theorem''.
Paige also proved a number of other results, including that every computed Ritz value lies nearly within the interval containing the eigenvalues of $\vec{A}$ and sufficient conditions for Ritz values to be near to eigenvalues.

Greenbaum later proved a backward-stability-type result for Lanczos \cite{greenbaum1989}.
Building on the work of Paige, she showed that the computed tridiagonal matrix can also be obtained as the output of Lanczos run in exact arithmetic on a nearby problem instance.
The nearby problem instance is of a larger dimension than $n$, but is close in the sense that its eigenvalues lie within small intervals about the eigenvalues of $\vec{A}$, and the spectral weights of its starting vector on each such interval are close to those of the original starting vector \cite[\S8]{greenbaum1989}; see also \cite{greenbaum_strakos1992}.

Several resources exist for learning about these results.
The survey of Meurant and Strako\v{s} \cite{meurant_strakos2006} is an accessible guide to the results, their history, and their consequences, but it states the central theorems without proof, deferring to the original papers.
Meurant's book \cite{meurant2006} collects these results and describes many aspects of their proofs carefully, but it likewise stops short of complete proofs of some key theorems.
Parlett \cite{parlett1998} distills a proof of Paige's Theorem to its essential mechanism, but makes use of heuristic arguments to do so.
The original papers themselves are long and notionally heavy; Greenbaum herself describes the theorems and proofs of \cite{greenbaum1989} as ``gory''.

\subsection{Aim and organization}

The aim of this note is to present a simplified, rigorous, self-contained analysis of results similar in spirit to the key results of Paige and Greenbaum.
Our main tools in this quest are two \emph{design choices}:
\begin{enumerate}[label=(\roman*)]
  \item suppress polynomial dependencies on the iteration count $k$, and
  \item assume linear-algebra kernels, rather than individual arithmetic operations, are performed accurately.
\end{enumerate}
While the resulting bounds are weaker than existing ones, we believe the tradeoff with simplicity is a good one.
First, re-deriving standard floating-point bounds for linear-algebra primitives from the standard model of floating point arithmetic sheds little light on the behavior of Lanczos itself.
Second, existing bounds for Lanczos are already numerically vacuous, even for relatively small iteration counts $k$ and dimensions $n$, in double precision arithmetic, so their value is in the intuition they provide rather than as a usable bound.
Third, polynomial dependencies on $k$ do not significantly impact the theoretical number of bits of precision needed, which grows only logarithmically with the target accuracy.

The remainder of this note is organized as follows.
\Cref{sec:model} describes the computational model and the computed recurrence.
\Cref{sec:local} shows that the computed run is well behaved locally, and derives a commutator relation for the Gram matrix of any vectors satisfying a perturbed Lanczos recurrence.
\Cref{sec:loss} deduces Paige's loss-of-orthogonality theorem from this relation for the computed Lanczos vectors.
\Cref{sec:backward} applies the same relation to the Lanczos vectors rescaled to unit length, uses it to construct a Greenbaum-type backward model, and deduces from it localization and containment of the computed Ritz values.

\subsection{Notation}

\begin{definition}[Polynomial convention]
Each occurrence of 
\[
 x\lesssim_k y
\]
is shorthand for \(x\le c k^dy\), where \(c,d>0\) are fixed absolute constants attached to that instance.
The constant and exponent may change from one occurrence to the next.
\end{definition}

\begin{remark}
This notation does not hide things that shouldn't be hidden such as the dimension, spectral quantities of $\vec{A}$, etc.
We provide a formal Lean proofs of all of main statements in this paper at \url{https://github.com/tchen-research/simple-lanczos-stability}. The Lean proofs include explicit constants and exponents.
\end{remark}

Throughout, \(\norm{\cdot}\) denotes the Euclidean norm of a vector and the induced operator norm of a matrix.
For a symmetric matrix \(\vec B\), \(\spec(\vec B)\) is its set of eigenvalues, and \(\lambda_{\min}(\vec B)\) and \(\lambda_{\max}(\vec B)\) are its smallest and largest eigenvalues.

\begin{algorithm}[htb]
\caption{Finite-precision symmetric Lanczos without reorthogonalization}
\label{alg:lanczos}
\small
\begin{algorithmic}[1]
\Require a symmetric matrix \(\vec A\); a nonzero starting vector
  \(\vec v\); a maximum iteration count
  \(k_{\max}\).  The calls \(\fl\) and \textsc{normalize} obey the
  assumptions of \cref{sec:model}.
\State \((\,\cdot\,,\vec v_1)\gets\textsc{normalize}(\vec v)\)
  \Comment{the scalar output is discarded}
\State \(\vec v_0\gets\vec0\), \(\beta_0\gets0\)
\For{\(j=1,2,\ldots,k_{\max}\)}
  \State \(\vec u_j\gets\fl(\vec A\vec v_j)\)
  \State \(\vec w_j\gets\fl(\vec u_j-\beta_{j-1}\vec v_{j-1})\)
  \State \(\alpha_j\gets\fl(\vec v_j^\T\vec w_j)\)
  \State \(\vec z_j\gets\fl(\vec w_j-\alpha_j\vec v_j)\)
  \If{\(\vec z_j=\vec0\)}
    \State \(\beta_j\gets0\); \textbf{terminate}
  \Else
    \State \((\beta_j,\vec v_{j+1})\gets\textsc{normalize}(\vec z_j)\)
  \EndIf
\EndFor
\Ensure 
  \(\Vk=[\vec v_1,\ldots,\vec v_k]\) and the tridiagonal \(\Tk\)
  formed from \(\alpha_1,\ldots,\alpha_k\) and
  \(\beta_1,\ldots,\beta_{k-1}\).
\end{algorithmic}
\end{algorithm}

\section{Setup}
\label{sec:model}

\subsection{Arithmetic and normalization}

We assume linear-algebra kernels can be performed to relative accuracy quantified by the accuracy parameter $\varepsilon$.
The precise form of our assumptions is based on the bounds one obtains from the standard model of floating point arithmetic \cite{higham2002}. 
As such, these assumptions can be ensured by setting the machine precision sufficiently small relative to the ambient dimension $n$.

\begin{assumption}[Arithmetic kernels]
\label{ass:arithmetic}
For compatible vectors and scalars, the computed kernels satisfy
\begin{align*}
 \fl(\vec x^\T\vec y)
 &=\vec x^\T\vec y+e,
 &\abs e&\le\eps\norm{\vec x}\norm{\vec y},
 \\
 \fl(\vec x\pm\alpha\vec y)
 &=\vec x\pm\alpha\vec y+\vec e,
 &\norm{\vec e}&\le\eps
   \bigl(\norm{\vec x}+\abs\alpha\norm{\vec y}\bigr),
 \\
 \fl(\vec A\vec v)
 &=\vec A\vec v+\vec e,
 &\norm{\vec e}&\le\eps\norm{\vec A}\norm{\vec v}.
\end{align*}
Copying and transposition are exact.
\end{assumption}

\begin{assumption}[Normalization]
\label{ass:normalization}
For \(\vec y\ne\vec0\), the call \((\beta,\vec v)=\textnormal{\textsc{normalize}}(\vec y)\) returns \(\beta>0\) and \(\vec v\) such that
\begin{align*}
 \beta\vec v&=\vec y+\bm\delta,
 &\norm{\bm\delta}&\le\eps\norm{\vec y},
 \\
 \vec v^\T\vec v &= 1+\delta & |\delta|&\le\eps.
\end{align*}
If the vector to be normalized is zero, the algorithm terminates before making the call.
\end{assumption}

Finally, we assume the accuracy parameter $\varepsilon$ is small.
\begin{assumption}[Small \(\eps\)]
\label{ass:small-eps}
For a run of at most $k_{\max}$ iterations, the accuracy parameter satisfies $0\leq \varepsilon \leq 1/(100k_{\max})$.
\end{assumption}

\subsection{The computed recurrence}

We analyze Paige's computational variant A1 \cite{paige1976}, stated explicitly in \cref{alg:lanczos}. 
This is the now-standard implementation of Lanczos.

\begin{remark}
The initial call returns \(\vec v_1\) proportional to \(\vec v+\bm\delta_0\) with \(\norm{\bm\delta_0}\le\eps\norm{\vec v}\), by the first equation in \cref{ass:normalization}.
In particular \(\vec v_1\) satisfies the second equation in \cref{ass:normalization}.
\end{remark}

Fix a run of \(k\) completed iterations for which \(\beta_1,\ldots,\beta_{k-1}>0\); the last step may be nonterminal (\(\beta_k>0\)) or terminal (\(\beta_k=0\)).
Let
\[
 \Vk=[\vec v_1,\ldots,\vec v_k],
 \qquad
 \Tk=
 \begin{bmatrix}
  \alpha_1&\beta_1&&\\
  \beta_1&\alpha_2&\ddots&\\
  &\ddots&\ddots&\beta_{k-1}\\
  &&\beta_{k-1}&\alpha_k
 \end{bmatrix},
\]
and, following \cite{saad2011}, write
\[
 \Vkp=[\Vk,\vkp],
 \qquad
 \uTk=
 \begin{bmatrix}
  \Tk\\
  \beta_k\vec e_k^\T
 \end{bmatrix}
 \in\R^{(k+1)\times k},
\]
so that \(\Vkp\uTk=\Vk\Tk+\beta_k\vkp\vec e_k^\T\).
If the run terminates at step \(k\), then \(\vec z_k=\vec0\) and \(\beta_k=0\).
In that case \(\vkp\) is not formed, and we adopt the conventions \(\bm\delta_k=\vec0\) and \(\vkp=\vec0\), so that \(\beta_k\vkp=\vec0\).

We assume the computed quantities in \cref{alg:lanczos} satisfy \cref{ass:normalization,ass:arithmetic}.
That is, that
\begin{align}\allowdisplaybreaks
 \vec u_j
 &=\vec A\vec v_j+\vec e^u_j,
 &\norm{\vec e^u_j}
 &\le\eps\norm{\vec A}\norm{\vec v_j},
 \label{eq:local-u}\\
 \vec w_j
 &=\vec u_j-\beta_{j-1}\vec v_{j-1}+\vec e^w_j,
 &\norm{\vec e^w_j}
 &\le\eps\bigl(\norm{\vec u_j}
       +\beta_{j-1}\norm{\vec v_{j-1}}\bigr),
 \label{eq:local-w}\\
 \alpha_j
 &=\vec v_j^\T\vec w_j+e^\alpha_j,
 &\abs{e^\alpha_j}
 &\le\eps\norm{\vec v_j}\norm{\vec w_j},
 \label{eq:local-alpha}\\
 \vec z_j
 &=\vec w_j-\alpha_j\vec v_j+\vec e^z_j,
 &\norm{\vec e^z_j}
 &\le\eps\bigl(\norm{\vec w_j}
       +\abs{\alpha_j}\norm{\vec v_j}\bigr),
 \label{eq:local-z}\\
 \beta_j\vec v_{j+1}
 &=\vec z_j+\bm\delta_j,
 &\norm{\bm\delta_j}
 &\le\eps\norm{\vec z_j},
 \label{eq:local-normalize}\\
 \vec v_{j+1}^\T\vec v_{j+1}
 &=1+\delta_{j+1},
 &\abs{\delta_{j+1}}
 &\le\eps.
 \label{eq:local-unit}
\end{align}
At a terminal final step, \cref{eq:local-normalize} holds by the convention above rather than by a normalization call, and \cref{eq:local-unit} does not apply since \(\vkp\) is not formed.\footnote{It is unlikely that this occurs in finite precision arithmetic. Nevertheless, we must handle this case for a rigorous proof.}

\section{Local errors}
\label{sec:local}

\subsection{Local error bounds}

In this section, we prove that the Lanczos algorithm is well-behaved locally; i.e. the computed Lanczos vectors nearly satisfy a symmetric three-term recurrence, are nearly unit length, and that consecutive vectors are nearly orthogonal.
Such bounds appeared in Paige's early work; see e.g. \cite[\S2-3]{paige1976}.
Specifically, we bound the following quantities:
\begin{equation}
  \begin{aligned}
     \vec f_j&:=\vec A\vec v_j-\beta_{j-1}\vec v_{j-1}-\alpha_j\vec v_j-\beta_j\vec v_{j+1},
 \\
 g_j&:=\vec v_j^\T\vec v_j-1,
 \\
 p_j&:=\vec v_j^\T(\beta_j\vec v_{j+1}).
  \end{aligned}
 \label{eq:derived-quantities}
\end{equation}
In matrix form, with \(\Fk:=[\vec f_1,\ldots,\vec f_k]\), the computed vectors satisfy the perturbed block recurrence
\begin{equation}
 \vec A\Vk=\Vkp\uTk+\Fk.
 \label{eq:governing}
\end{equation}

We begin with a simple lemma showing that none of the quantities of relevance blow up.

\begin{lemma}[A priori coefficient bounds]
\label{lem:coefficients}
For \(1\le j\le k\),
\[
 \max\{\norm{\vec u_j},\norm{\vec w_j},\abs{\alpha_j},
          \norm{\vec z_j},\beta_j\}
 \lesssim_k\norm{\vec A}.
\]
\end{lemma}

\begin{proof}
The main difficulty is that the quantities at step $j$ depend on the quantities at step $j-1$.
We therefore work relative to \(\norm{\vec w_j}\) until the very end.
Step 1 expresses \(\abs{\alpha_j}\), \(\norm{\vec z_j}\), and \(\beta_j\) as multiples of \(\norm{\vec w_j}\) with absolute constants, and Step 2 bounds \(\norm{\vec w_j}\) in terms of \(\beta_{j-1}\).
Composing the two eliminates \(\norm{\vec w_j}\) and leaves a scalar recurrence in \(\beta_j\) alone. We solve this recurrence in Step 3, and substitute the solution back into the recurrence in Step 4 to get the final bounds.

\emph{Step 1a: \(\norm{\vec z_j}\) relative to \(\norm{\vec w_j}\).}
Substituting \cref{eq:local-alpha} into \cref{eq:local-z} gives
\[
 \vec z_j=
 \underbrace{\vec w_j-\vec v_j(\vec v_j^\T\vec w_j)}_
             {\text{exact-arithmetic update}}
 -e_j^\alpha\vec v_j+\vec e_j^z.
\]
The braced rank-one update above does not increase the norm, even though \(\vec v_j\) is only nearly normalized.
Indeed, by \cref{eq:local-unit}, \(\norm{\vec v_j}^2\le1+\eps<2\).
Therefore, direct expansion gives
\[
 \norm{\vec w_j-\vec v_j(\vec v_j^\T\vec w_j)}^2
 =\norm{\vec w_j}^2
  -\bigl(2-\vec v_j^\T\vec v_j\bigr)
   (\vec v_j^\T\vec w_j)^2
 \le\norm{\vec w_j}^2.
\]
The two remaining terms are handled next.
By \cref{eq:local-alpha}, Cauchy--Schwarz, and \cref{eq:local-unit},
\[
 \abs{\alpha_j}\norm{\vec v_j}
 \le(1+\eps)\norm{\vec v_j}^2\norm{\vec w_j}
 \le(1+\eps)^2\norm{\vec w_j},
 \quad
 \abs{e_j^\alpha}\norm{\vec v_j}
 \le\eps(1+\eps)\norm{\vec w_j}.
\]
Using this and \cref{eq:local-z}, we have
\[
 \norm{\vec e_j^z}
 \le\eps\bigl(\norm{\vec w_j}+\abs{\alpha_j}\norm{\vec v_j}\bigr)
 \le\eps\bigl(1+(1+\eps)^2\bigr)\norm{\vec w_j}.
\]
Applying the triangle inequality to the three terms of the displayed identity for \(\vec z_j\), and inserting the three bounds just obtained, gives
\[
 \norm{\vec z_j}
 \le\bigl[1+\eps(1+\eps)
              +\eps\bigl(1+(1+\eps)^2\bigr)\bigr]\norm{\vec w_j}
 \le(1+5\eps)\norm{\vec w_j},
\]
where the last inequality holds because \(\eps\le1/100\), by \cref{ass:small-eps}.

\emph{Step 1b: \(\beta_j\) relative to \(\norm{\vec w_j}\).}
At a nonterminal step, \cref{eq:local-normalize} and \(\norm{\vec v_{j+1}}\ge\sqrt{1-\eps}\), from \cref{eq:local-unit}, give
\[
 \beta_j
 =\frac{\norm{\vec z_j+\bm\delta_j}}{\norm{\vec v_{j+1}}}
 \le\frac{1+\eps}{\sqrt{1-\eps}}\norm{\vec z_j}
 \le(1+2\eps)\norm{\vec z_j}
 \le (1+2\eps)(1+5\eps)\norm{\vec w_j},
\]
where we have again used \(\eps\le1/100\) and our bound on $\norm{\vec z_j}$.

\emph{Step 2: \(\norm{\vec w_j}\) in terms of \(\beta_{j-1}\).}
The bound on \(\vec u_j\) is immediate from \cref{eq:local-u,eq:local-unit}:
\[
 \norm{\vec u_j}
 \le(1+\eps)\norm{\vec A}\norm{\vec v_j}
 \le(1+\eps)^{3/2}\norm{\vec A}
 \le2\norm{\vec A}.
\]
Feeding it into \cref{eq:local-w} is where the previous iteration enters:
\[
 \norm{\vec w_j}
 \le(1+\eps)\bigl(\norm{\vec u_j}
                    +\beta_{j-1}\norm{\vec v_{j-1}}\bigr)
 \le(1+\eps)^{5/2}\norm{\vec A}+(1+\eps)^{3/2}\beta_{j-1}.
\]

\emph{Step 3: closing the recurrence.}
We now plug the bound on $\norm{\vec w_j}$ into the bound for $\beta_j$:
\[
 \begin{aligned}
  \beta_j
  &\le(1+2\eps)(1+5\eps)
       \bigl((1+\eps)^{5/2}\norm{\vec A}
             +(1+\eps)^{3/2}\beta_{j-1}\bigr)\\
  &\le(1+10\eps)\beta_{j-1}+2\norm{\vec A},
 \end{aligned}
\]
where the last line uses the two elementary inequalities
\[
 (1+2\eps)(1+5\eps)(1+\eps)^{3/2}\le1+10\eps,
 \qquad
 (1+2\eps)(1+5\eps)(1+\eps)^{5/2}\le2,
\]
both valid for \(0\le\eps\le1/100\).

We now prove \(\beta_j\le4j\norm{\vec A}\) by induction.
By definition, \(\beta_0=0\).
Then, assuming the bound at \(j-1\),
\[
 \beta_j
 \le(1+10\eps)\,4(j-1)\norm{\vec A}+2\norm{\vec A}
 =\bigl(4(j-1)+40(j-1)\eps+2\bigr)\norm{\vec A}
 \le4j\norm{\vec A},
\]
because \(40(j-1)\eps\le40k_{\max}\eps\le2\) by \cref{ass:small-eps}.

\emph{Step 4: back-substitution.}
With \(\beta_j\) now bounded absolutely, the remaining relative bounds become absolute.
Substituting \(\beta_{j-1}\le4(j-1)\norm{\vec A}\) into the Step 2 bound on \(\norm{\vec w_j}\), then that into the Step 1 bounds, dividing the bound on \(\abs{\alpha_j}\norm{\vec v_j}\) by \(\norm{\vec v_j}\ge\sqrt{1-\eps}\), and again using \(\eps\le1/100\), gives
\[
 \norm{\vec w_j}\le5j\norm{\vec A},
 \qquad
 \abs{\alpha_j}\le6j\norm{\vec A},
 \qquad
 \norm{\vec z_j}\le7j\norm{\vec A}.
\]
Together with \(\norm{\vec u_j}\le2\norm{\vec A}\) from Step 2 and \(\beta_j\le4j\norm{\vec A}\) from Step 3, every quantity in the statement is at most \(7j\norm{\vec A}\), which proves the claimed bounds.
At a terminal final step \(\beta_k=0\), and the same bounds apply using the already controlled \(\beta_{k-1}\).
\end{proof}

We now prove our main local bound.

\begin{theorem}[Local error bounds]
\label{thm:local-errors}
For \(1\le j\le k\),
\[
 \norm{\vec f_j}\lesssim_k\eps\norm{\vec A},
 \qquad
 \abs{g_j}\le\eps,
 \qquad
 \abs{p_j}\lesssim_k\eps\norm{\vec A}.
\]
\end{theorem}

\begin{proof}
The bound on \(g_j\) is \cref{eq:local-unit} for \(j\ge2\), and the initial call gives the same bound for \(g_1\).

For \(\vec f_j\), eliminate the intermediate vectors: substitute \cref{eq:local-u} into \cref{eq:local-w}, the result into \cref{eq:local-z}, and that into \cref{eq:local-normalize}.
Plugging that into the definition of \(\vec f_j\) in \cref{eq:derived-quantities} gives
\[
 \vec f_j=-(\vec e^u_j+\vec e^w_j+\vec e^z_j+\bm\delta_j).
\]
Next, using the triangle inequality and \cref{eq:local-u,eq:local-w,eq:local-z,eq:local-normalize}:
\begin{align*}
   \norm{\vec f_j}
 &\le\norm{\vec e^u_j}+\norm{\vec e^w_j}
   +\norm{\vec e^z_j}+\norm{\bm\delta_j}
 \\&\le\eps\bigl(\norm{\vec A}\norm{\vec v_j}
   +\norm{\vec u_j}+\beta_{j-1}\norm{\vec v_{j-1}}
   +\norm{\vec w_j}+\abs{\alpha_j}\norm{\vec v_j}
   +\norm{\vec z_j}\bigr).
\end{align*}
By \cref{eq:local-unit} each \(\norm{\vec v_i}\) is at most \(\sqrt{1+\eps}\), and by \cref{lem:coefficients} each of \(\norm{\vec u_j}\), \(\norm{\vec w_j}\), \(\abs{\alpha_j}\), \(\norm{\vec z_j}\), and \(\beta_{j-1}\) is \(\lesssim_k\norm{\vec A}\), which gives the bound on $\norm{\vec f_j}$

For \(p_j\), \crefrange{eq:local-alpha}{eq:local-normalize} give
\[
 p_j
 =\vec v_j^\T\vec z_j+\vec v_j^\T\bm\delta_j
 =\vec v_j^\T\vec w_j-\alpha_j(1+g_j)
   +\vec v_j^\T\vec e^z_j+\vec v_j^\T\bm\delta_j.
\]
Since \(\vec v_j^\T\vec w_j=\alpha_j-e^\alpha_j\), the two exact copies of \(\alpha_j\) cancel, leaving
\[
 p_j=-e^\alpha_j-g_j\alpha_j
      +\vec v_j^\T\vec e^z_j+\vec v_j^\T\bm\delta_j.
\]
Finally, use \cref{eq:local-alpha,eq:local-z,eq:local-normalize,eq:local-unit} and \cref{lem:coefficients} to obtain the bound on \(p_j\).
\end{proof}

\subsection{The Gram commutator}

Decompose the Gram matrix of the computed Lanczos vectors as
\begin{equation}
 \Vk^\T\Vk=\vec I+\Rk^\T+\diag(g_1,\ldots,g_k)+\Rk,
 \qquad
 \Rkp:=
 \begin{bmatrix}
  \Rk&\Vk^\T\vkp\\
  \vec 0^\T&0
 \end{bmatrix},
 \label{eq:gram-split}
\end{equation}
where \(\Rk\) is strictly upper triangular, so that \(\Rkp\) is the strict upper triangle of \(\Vkp^\T\Vkp\).
Since \(\vec e_k^\T\Rk=\vec0\), the last row of \(\uTk\Rk-\Rkp\uTk\) vanishes, and we define \(\Mk\) by
\begin{equation}
 \uTk\Rk-\Rkp\uTk
 =\begin{bmatrix}\Mk\\\vec 0^\T\end{bmatrix},
 \qquad
 \Mk:=\Tk\Rk-\Rk\Tk-\beta_k(\Vk^\T\vkp)\vec e_k^\T.
 \label{eq:M-def}
\end{equation}
A key observation of Paige's analysis \cite[eq.~(22)]{paige1976}, \cite[eqs.~(2.17),(2.18)]{paige1980} is that \(\Mk\) is small, that is, \(\Rkp\uTk\approx\uTk\Rk\).
Thus, while \(\Rk\) may be far from zero, as it would be in finite-precision arithmetic, it cannot be arbitrary.

\begin{corollary}
\label{cor:gram-commutator}
The matrix \(\Mk\) in \cref{eq:M-def} is upper triangular, and \(\norm{\Mk}\lesssim_k\eps\norm{\vec A}\).
\end{corollary}

\Cref{cor:gram-commutator} will follow immediately from the following lemma, which we will use again later.

\begin{lemma}
\label{lem:gram-commutator}
Let \(\Wkp=[\Wk,\vec w_{k+1}]\in\R^{n\times(k+1)}\) have columns \(\vec w_1,\ldots,\vec w_{k+1}\), and let \(\vec Z\in\R^{n\times k}\) satisfy \(\vec A\Wk=\Wkp\uTk+\vec Z\).
Decompose \(\Wk^\T\Wk=\vec I+\Yk^\T+\bm\Gamma+\Yk\) with \(\Yk\) strictly upper triangular and \(\bm\Gamma=\diag(\gamma_1,\ldots,\gamma_k)\), let \(\Ykp\) be the strict upper triangle of \(\Wkp^\T\Wkp\), and define \(\Nk\) by
\[
 \uTk\Yk-\Ykp\uTk
 =\begin{bmatrix}\Nk\\\vec 0^\T\end{bmatrix}.
\]
Then \(\Nk\) is upper triangular, and with \(\pi_i:=\beta_i\vec w_i^\T\vec w_{i+1}\) for \(i\le k\),
\[
 \norm{\Nk}
 \le2k\Bigl(\max_{i\le k}\abs{\pi_i}
 +\max_{l\le k}\norm{\vec w_l}\,\norm{\vec Z}
 +\max_{i<k}\beta_i\abs{\gamma_{i+1}-\gamma_i}\Bigr).
\]
\end{lemma}

\begin{proof}
As for \(\Mk\) in \cref{eq:M-def}, the last row of \(\uTk\Yk-\Ykp\uTk\) vanishes and
\[
 \Nk=\Tk\Yk-\Yk\Tk-\beta_k(\Wk^\T\vec w_{k+1})\vec e_k^\T.
\]
Since \(\Tk\) is tridiagonal and \(\Yk\) is strictly upper triangular, \(\Tk\Yk\) and \(\Yk\Tk\) are upper triangular, and the last term is supported on the last column; hence \(\Nk\) is upper triangular.

Left multiplying the recurrence for \(\Wk\) by \(\Wk^\T\), and using \(\Wkp\uTk=\Wk\Tk+\beta_k\vec w_{k+1}\vec e_k^\T\), gives
\[
 \Wk^\T\vec A\Wk
 =(\Wk^\T\Wk)\Tk+\beta_k(\Wk^\T\vec w_{k+1})\vec e_k^\T+\Wk^\T\vec Z.
\]
The left side is symmetric, so equating the right side with its transpose gives the symmetric form of the identity,
\begin{equation}
 \Tk(\Wk^\T\Wk)-(\Wk^\T\Wk)\Tk
 =\beta_k(\Wk^\T\vec w_{k+1})\vec e_k^\T-\beta_k\vec e_k(\Wk^\T\vec w_{k+1})^\T
 +\Wk^\T\vec Z-\vec Z^\T\Wk.
 \label{eq:gram-commutator-full}
\end{equation}
Now, substituting \(\Wk^\T\Wk=\vec I+\Yk^\T+\bm\Gamma+\Yk\) into \cref{eq:gram-commutator-full} and using the expression for \(\Nk\),
\begin{equation}
 \Nk-\Nk^\T
 =\Wk^\T\vec Z-\vec Z^\T\Wk-(\Tk\bm\Gamma-\bm\Gamma\Tk).
 \label{eq:generic-commutator}
\end{equation}

It remains to bound the entries of \(\Nk\).
For the diagonal,
\[
 (\Tk\Yk)_{ii}
 =\beta_{i-1}\vec w_{i-1}^\T\vec w_i=\pi_{i-1},
 \qquad
 (\Yk\Tk)_{ii}
 =\beta_i\vec w_i^\T\vec w_{i+1}=\pi_i,
\]
where the first is zero for \(i=1\) and the second is absent for \(i=k\), because \(\Yk\) has no column \(k+1\); in its place, the last term of \(\Nk\) contributes \(\beta_k\vec w_k^\T\vec w_{k+1}=\pi_k\).
Hence every diagonal entry of \(\Nk\) has magnitude at most \(2\max_i\abs{\pi_i}\).
Since \(\Nk^\T\) is lower triangular, every strict upper entry of \(\Nk\) equals the corresponding entry of the right side of \cref{eq:generic-commutator}.
For \(i<j\),
\[
 \abs{(\Wk^\T\vec Z-\vec Z^\T\Wk)_{ij}}
 =\abs{\vec w_i^\T(\vec Z\vec e_j)-(\vec Z\vec e_i)^\T\vec w_j}
 \le2\max_{l\le k}\norm{\vec w_l}\,\norm{\vec Z},
\]
and the last term in \cref{eq:generic-commutator} is supported on the first super- and subdiagonals, with \((\Tk\bm\Gamma-\bm\Gamma\Tk)_{i,i+1}=\beta_i(\gamma_{i+1}-\gamma_i)\).
Thus every entry of \(\Nk\) is bounded by the parenthesized quantity in the lemma times \(2\), and since \(\Nk\) is \(k\times k\), passing through the Frobenius norm gives the bound on \(\norm{\Nk}\).
\end{proof}

\begin{proof}[Proof of \cref{cor:gram-commutator}]
We apply \cref{lem:gram-commutator} with \((\Wkp,\vec Z)=(\Vkp,\Fk)\), which is allowed by \cref{eq:governing}.
By \cref{eq:gram-split}, \(\Yk=\Rk\), \(\Ykp=\Rkp\), \(\gamma_j=g_j\), \(\Nk=\Mk\), and \(\pi_i=p_i\).
Finally, \cref{thm:local-errors} gives \(\abs{p_i}\lesssim_k\eps\norm{\vec A}\), \(\norm{\vec v_l}\le\sqrt{1+\eps}\), and \(\norm{\Fk}\le\sqrt{k}\,\max_j\norm{\vec f_j}\lesssim_k\eps\norm{\vec A}\), and \cref{thm:local-errors,lem:coefficients} give \(\beta_i\abs{g_{i+1}-g_i}\le2\eps\beta_i\lesssim_k\eps\norm{\vec A}\).
\end{proof}

\section{Loss of orthogonality}
\label{sec:loss}

Let \((\theta,\vec y)\) be a unit eigenpair of \(\Tk\), and define the corresponding Ritz vector \(\vec x:=\Vk\vec y\).
Because \(\Tk\vec y=\theta\vec y\), the perturbed Lanczos recurrence relation \cref{eq:governing} gives the exact residual formula
\begin{equation}
 (\vec A-\theta\vec I)\vec x
 =(\beta_k\vkp)(\vec e_k^\T\vec y)+\Fk\vec y.
 \label{eq:ritz-residual-identity}
\end{equation}
Taking norms in \cref{eq:ritz-residual-identity}, and using \cref{eq:local-unit}, we find
\begin{equation}
 \frac{\norm{(\vec A-\theta\vec I)\vec x}}{\norm{\vec x}}
 \le \frac{\sqrt{1+\eps}\,\beta_k\abs{\vec e_k^\T\vec y}
      +\norm{\Fk}}{\norm{\vec x}}.
 \label{eq:ritz-residual-bound}
\end{equation}
In exact arithmetic, $\Vk$ has orthonormal columns, and hence $\norm{\vec x}=1$.
In this case, $\vec{x}$ is nearly a unit eigenvector of $\vec{A}$ with approximate eigenvalue $\theta$ when $\beta_k\abs{\vec e_k^\T\vec y}$ is small.
However, a loss of orthogonality can make \(\norm{\vec x}\) small, in which case $\beta_k\abs{\vec e_k^\T\vec y}$ being small need not imply $\vec{x}/\norm{\vec x}$ is nearly an eigenvector.

The next theorem is Paige's famous loss-of-orthogonality bound \cite[eqs.~(3.11)-(3.13)]{paige1980}, that asserts that a loss of orthogonality against a Ritz direction can occur only when its scalar residual estimate is small; that is, by \cref{eq:ritz-residual-bound}, only when the Ritz pair has nearly converged to an eigenpair of $\vec{A}$.
The same tradeoff is visible numerically in \cref{fig:paige-dichotomy}.

\begin{figure}[t]
\centering
\includegraphics[scale=.86]{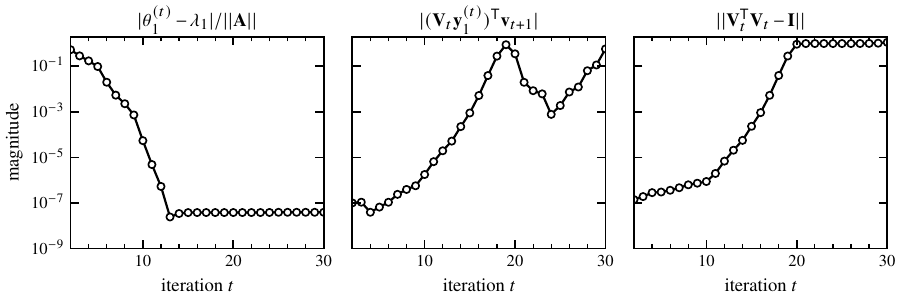}
\caption{Convergence and loss of orthogonality for the Ritz approximation to the largest eigenvalue \(\lambda_1=1\), in the same single-precision Lanczos run as \cref{fig:backward-experiment}.
At iteration \(t\), let \(\vec V_t\) and \(\vec T_t\) be the Lanczos basis and tridiagonal matrix of the first \(t\) iterations, and let \((\theta_1^{(t)},\vec y_1^{(t)})\) be the largest unit eigenpair of \(\vec T_t\).
Left: the normalized Ritz-value error \(\abs{\theta_1^{(t)}-\lambda_1}/\norm{\vec A}\).
  Center: the corresponding orthogonality defect \(\abs{(\vec V_t\vec y_1^{(t)})^\T\vec v_{t+1}}\).
  Right: the global loss of orthogonality \(\norm{\vec V_t^\T\vec V_t-\vec I}\).
  The Ritz value reaches the rounding-error floor around iteration 13, after which both measures of lost orthogonality grow to order one, consistent with the convergence--orthogonality dichotomy of \cref{thm:paige}.}
\label{fig:paige-dichotomy}
\end{figure}

\begin{theorem}[Paige loss-of-orthogonality bound]
\label{thm:paige}
For every unit eigenpair \((\theta,\vec y)\) of \(\Tk\), the Ritz vector \(\vec x=\Vk\vec y\) satisfies
\[
 \abs{\vec x^\T\vec{v}_{k+1}}\cdot\beta_k\abs{\vec e_k^\T\vec y}
 \lesssim_k\eps\norm{\vec A}.
\]
\end{theorem}

\begin{proof}
Fix a unit eigenpair \((\theta,\vec y)\), and let \(\Mk\) be as in \cref{eq:M-def}.
Since \(\Tk\) is symmetric and \(\Tk\vec y=\theta\vec y\), the commutator in \(\Mk\) vanishes in the quadratic form: \(\vec y^\T(\Tk\Rk-\Rk\Tk)\vec y=\theta\,\vec y^\T\Rk\vec y-\theta\,\vec y^\T\Rk\vec y=0\).
The remaining term in \cref{eq:M-def} is \(\beta_k(\vec y^\T\Vk^\T\vkp)(\vec e_k^\T\vec y)=\beta_k(\vec x^\T\vkp)(\vec e_k^\T\vec y)\).
Hence, by \cref{cor:gram-commutator},
\[
 \abs{\vec x^\T\vkp}\cdot\beta_k\abs{\vec e_k^\T\vec y}
 =\abs{\vec y^\T\Mk\vec y}
 \le\norm{\Mk}
 \lesssim_k\eps\norm{\vec A}.\qedhere
\]
\end{proof}

\section{A backward model}
\label{sec:backward}

We now show that the computed \(\Tk\) is the exact output of \(k\) steps of Lanczos applied to a symmetric matrix that is close to \(\vec I_{k+1}\otimes\vec A\), a backward interpretation of the kind introduced by Greenbaum in \cite{greenbaum1989}.
Our construction differs from Greenbaum's.
Whereas Greenbaum builds the model by extending the computed tridiagonal matrix, continuing the recurrence in exact arithmetic and tracking where the extended eigenvalues fall \cite[\S5-7]{greenbaum1989}, we instead build an orthonormal representation of the computed Lanczos vectors by a dilation into \(k+1\) stacked copies of \(\R^n\), and then make a small symmetric correction so that the dilated vectors satisfy the computed recurrence exactly.

Our construction is closest to Paige's augmented stability result \cite{paige2010}.
Using his unitary augmentation of a sequence of unit vectors \cite{paige2009}, Paige showed that the computed \(\Tk\) is the exact output of \(k\) steps of Lanczos applied to a small perturbation of \(\diag(\Tk,\vec A)\), of dimension \(n+k\).
Since \(\Tk\) itself appears in this augmented matrix, some of its eigenvalues need not lie near \(\spec(\vec A)\) \cite[\S3.3]{paige2010}.
Our dilation iterates Paige's augmentation until the block involving \(\Tk\) vanishes (\cref{rem:dilation-origin}), which replaces \(\diag(\Tk,\vec A)\) by \(\vec I_{k+1}\otimes\vec A\).

Our approach has several payoffs: (i) the model is built in closed form, in dimension exactly \(n(k+1)\); (ii) its perturbation is of the same order \(\eps\norm{\vec A}\) as the local error bounds (Greenbaum's bound depends on the fourth-root of the accuracy parameter \cite[eq.~(7.10)]{greenbaum1989}); (iii) our backward model construction uses only the local analysis of \cref{sec:local}, and not more advanced results of \cref{thm:paige}.
Paige's localization and containment results for the computed Ritz values follow from our approach as short corollaries (\cref{sec:ritz-consequences}).

\subsection{Approach sketch}
\label{sec:backward-sketch}

Let \(\hA:=\vec I_{k+1}\otimes\vec A\).
First, in \cref{sec:dilation}, we explicitly construct \(\hVkp=[\hVk,\hvkp]\in\R^{n(k+1)\times(k+1)}\) with orthonormal columns such that
\[
 \hA\hVk=\hVkp\uTk+\hE,
 \qquad
 \norm{\hE}\lesssim_k\eps\norm{\vec A}.
\]
This is the computed recurrence \cref{eq:governing}, with the same \(\uTk\), but with exactly orthonormal Lanczos vectors.
Moreover, the projection of \(\hVk\vec e_1\) onto each of the eigenspaces of \(\hA\) matches that of \(\bar{\vec v}_1\) onto the corresponding eigenspace of \(\vec A\).

Second, in \cref{sec:correction}, we remove the residual \(\hE\) by a symmetric perturbation \(\bm\Delta\) with \(\norm{\bm\Delta}\le2\norm{\hE}\), so that \(\tA:=\hA+\bm\Delta\) satisfies the recurrence exactly:
\[
 \tA\hVk=\hVkp\uTk.
\]
Thus \(k\) steps of exact Lanczos applied to \(\tA\) from \(\tv:=\hVk\vec e_1\) produce exactly the computed \(\Tk\) and \(\beta_k\).
Since \(\tA\) is close to \(\hA\), whose spectrum is that of \(\vec A\), and \(\tv\) has the same spectral weights with respect to \(\hA\) as \(\bar{\vec v}_1\) has with respect to \(\vec A\), this a ``nearby problem''.

\subsection{Normalized Lanczos vectors}
\label{sec:normalized}

The construction of the backward model needs the Lanczos vectors rescaled to exactly unit length.
Let \(d_j:=\norm{\vec v_j}\) and \(\bar{\vec v}_j:=\vec v_j/d_j\) for \(1\le j\le k+1\); at a terminal step, where \(\beta_k=0\) and \(\vkp=\vec0\), let instead \(\bvkp\) be any unit vector and \(d_{k+1}:=1\).
In parallel with \cref{eq:derived-quantities}, define
\begin{equation}
  \begin{aligned}
     \bar{\vec f}_j&:=\vec A\bar{\vec v}_j-\beta_{j-1}\bar{\vec v}_{j-1}-\alpha_j\bar{\vec v}_j-\beta_j\bar{\vec v}_{j+1},
 \\
 \bar p_j&:=\bar{\vec v}_j^\T(\beta_j\bar{\vec v}_{j+1}).
  \end{aligned}
 \label{eq:normalized-quantities}
\end{equation}
In matrix form, with \(\bVkp:=[\bar{\vec v}_1,\ldots,\bvkp]=[\bVk,\bvkp]\) and \(\bFk:=[\bar{\vec f}_1,\ldots,\bar{\vec f}_k]\), the normalized vectors satisfy the perturbed block recurrence
\begin{equation}
 \vec A\bVk=\bVkp\uTk+\bFk.
 \label{eq:normalized-interface}
\end{equation}

We now prove an analog of \cref{thm:local-errors} for the normalized quantities.
\begin{lemma}[Normalized local errors]
\label{lem:normalized-errors}
For \(1\le j\le k\),
\[
 \norm{\bar{\vec f}_j}\lesssim_k\eps\norm{\vec A},
 \qquad
 \abs{\bar p_j}\lesssim_k\eps\norm{\vec A}.
\]
\end{lemma}

\begin{proof}
For \(j\le k\), \(d_j^2=1+g_j\) with \(\abs{g_j}\le\eps\) by \cref{thm:local-errors}, and if \(\beta_k>0\), \(d_{k+1}^2=1+\delta_{k+1}\) with \(\abs{\delta_{k+1}}\le\eps\) by \cref{eq:local-unit}.
Since \(\eps\le3/4\) by \cref{ass:small-eps}, every \(d_j\) with \(j\le k+1\) satisfies \(d_j\ge1/2\), \(\abs{d_j-1}\le\abs{d_j^2-1}\le\eps\), and hence \(\abs{1/d_j-1}\le2\eps\).
Comparing \cref{eq:normalized-quantities} with \cref{eq:derived-quantities} term by term,
\[
 \bar{\vec f}_j-\frac{\vec f_j}{d_j}
 =\beta_{j-1}\vec v_{j-1}\Bigl(\frac1{d_j}-\frac1{d_{j-1}}\Bigr)
 +\beta_j\vec v_{j+1}\Bigl(\frac1{d_j}-\frac1{d_{j+1}}\Bigr),
\]
where the first term is absent for \(j=1\), and the second vanishes at a terminal step.
Each difference of reciprocals is at most \(4\eps\), each \(\norm{\vec v_i}\le\sqrt{1+\eps}\), and each \(\beta_i\lesssim_k\norm{\vec A}\) by \cref{lem:coefficients}, so the right side is \(\lesssim_k\eps\norm{\vec A}\); together with \(\norm{\vec f_j/d_j}\le2\norm{\vec f_j}\) and \cref{thm:local-errors}, this bounds \(\bar{\vec f}_j\).
For the products, \(\bar p_j=p_j/(d_jd_{j+1})\), so \(\abs{\bar p_j}\le4\abs{p_j}\lesssim_k\eps\norm{\vec A}\) by \cref{thm:local-errors}.
\end{proof}

In parallel with \cref{eq:gram-split,eq:M-def}, decompose the normalized Gram matrix as
\begin{equation}
 \bGkp:=\bVkp^\T\bVkp=\vec I+\bRkp^\T+\bRkp,
 \label{eq:normalized-gram}
\end{equation}
where \(\bRkp\) is strictly upper triangular with leading \(k\times k\) block \(\bRk\), and the diagonal is exactly the identity because the columns are exactly unit.
As in \cref{eq:M-def}, define \(\bMk\) by
\begin{equation}
 \uTk\bRk-\bRkp\uTk
 =\begin{bmatrix}\bMk\\\vec 0^\T\end{bmatrix}.
 \label{eq:normalized-M-def}
\end{equation}

We also have an analog of \cref{cor:gram-commutator}.
\begin{corollary}
\label{cor:normalized-commutator}
The matrix \(\bMk\) in \cref{eq:normalized-M-def} is upper triangular, and \(\norm{\bMk}\lesssim_k\eps\norm{\vec A}\).
\end{corollary}

\begin{proof}
By \cref{eq:normalized-interface,eq:normalized-gram}, \cref{lem:gram-commutator} applies with \((\Wkp,\vec Z)=(\bVkp,\bFk)\), and then \(\Yk=\bRk\), \(\Ykp=\bRkp\), \(\bm\Gamma=\vec0\), and \(\Nk=\bMk\).
The products are \(\pi_i=\bar p_i\), the columns \(\bar{\vec v}_l\) are unit vectors, and \(\norm{\bFk}\le\sqrt{k}\,\max_j\norm{\bar{\vec f}_j}\); all are controlled by \cref{lem:normalized-errors}.
\end{proof}

\subsection{Exact orthogonalization by a nilpotent dilation}
\label{sec:dilation}

The first step to our backwards model is to obtain a perturbed symmetric three-term recurrence, with the computed Lanczos coefficients, where the vectors in the recurrence are orthonormal.
Towards this end, set
\begin{equation}
 \Ckp:=(\vec I+\bRkp)^{-1},
 \qquad
 \Kkp:=\vec I-\Ckp=\bRkp\Ckp,
 \label{eq:CKsy}
\end{equation}
and let \(\Ck\) and \(\Kk\) be their leading \(k\times k\) blocks.
Since \(\bRkp\) is strictly upper triangular with leading block \(\bRk\), we have \(\Ck=(\vec I+\bRk)^{-1}\) and \(\Kk=\vec I-\Ck=\bRk\Ck\).
Let $\col(\cdots)$ stack matrices vertically, and define \(\hVkp\in\R^{n(k+1)\times(k+1)}\) by
\begin{equation}
 \hVkp
 := \col\bigl( \bVkp\Ckp, \bVkp\Ckp\Kkp,
  \ldots, \bVkp\Ckp(\Kkp)^{k}\bigr).
 \label{eq:dilation-def}
\end{equation}
Let \(\hVk\) be the first \(k\) columns of \(\hVkp\) and \(\hvkp\) the final column, so that \(\hVkp=[\hVk,\hvkp]\).

We first show that the dilated vectors are exactly orthonormal.
This is an algebraic fact about the construction of $\hVkp$ that relies on the fact that $\bVkp$ has normalized columns.

\begin{lemma}
\label{lem:dilation}
It holds that $(\hVkp)^\T \hVkp=\vec I$ and $\hVkp\vec e_1=\col(\bar{\vec v}_1,\vec0,\ldots,\vec0)$.
\end{lemma}

\begin{proof}
Write \(\hVkp^{(r)}:=\bVkp\Ckp(\Kkp)^r\) for the \(r\)th block of \(\hVkp\), \(0\le r\le k\).
Since \(\bRkp\) is strictly upper triangular, \(\Ckp\) is unit upper triangular and \(\Kkp\) is strictly upper triangular.
In particular 
\begin{equation}
  (\Kkp)^{k+1}=\vec0,\qquad \Ckp\vec e_1=\vec e_1, \quad\text{and}\quad \Kkp\vec e_1=\vec0.
  \label{eq:Kkp-power}
\end{equation}
The last two facts give the result for the first column of \(\hVkp\).

Next, since \(\bRkp\Ckp=\Ckp\bRkp=\Kkp\) and \(\Ckp+\Kkp=\vec I\) by \cref{eq:CKsy}, the decomposition \cref{eq:normalized-gram} gives
\begin{align}
  (\Ckp)^\T\bGkp\Ckp
  &=(\Ckp)^\T\Ckp+(\Ckp)^\T\Kkp+(\Kkp)^\T\Ckp \nonumber
  \\
  &=\vec I-(\Kkp)^\T\Kkp.
 \label{eq:CGC-identity}
\end{align}
Hence, by \cref{eq:CGC-identity},
\[
 (\hVkp^{(r)})^\T\hVkp^{(r)}
 =\bigl((\Kkp)^r\bigr)^\T(\Kkp)^r-\bigl((\Kkp)^{r+1}\bigr)^\T(\Kkp)^{r+1}.
\]
Summing over the blocks gives a telescoping sum
\begin{align*}
 (\hVkp)^\T\hVkp
 &=\sum_{r=0}^{k}\bigl((\Kkp)^r\bigr)^\T(\Kkp)^r-\bigl((\Kkp)^{r+1}\bigr)^\T(\Kkp)^{r+1}
 \\
 &=\vec I-\bigl((\Kkp)^{k+1}\bigr)^\T(\Kkp)^{k+1}
 =\vec I,
\end{align*}
where the last equality follows by \cref{eq:Kkp-power}.
\end{proof}

\begin{remark}
\label{rem:dilation-origin}
By \cref{eq:CGC-identity}, \(\col(\Kkp,\bVkp\Ckp)\) has orthonormal columns.
This is Paige's unitary augmentation of the unit vectors \(\bar{\vec v}_1,\ldots,\bvkp\) \cite[Theorem~2.1]{paige2009}, in which \(\Kkp=(\vec I+\bRkp)^{-1}\bRkp\) appears as \(\vec S_{k+1}\) and \(\Ckp\) as \(\vec I-\vec S_{k+1}\); it is the source of the augmented model of \cite{paige2010}.
The dilation \cref{eq:dilation-def} repeats this augmentation: the block \(\Kkp\) is replaced by \(\col(\bVkp\Ckp\Kkp,(\Kkp)^2)\), which has the same Gram matrix by \cref{eq:CGC-identity}, and so on until \((\Kkp)^{k+1}=\vec0\).
The telescoping argument is the standard construction of an isometric dilation of a contraction \cite{schaffer1955,sznagy_foias_bercovici_kerchy2010}.
\end{remark}

The rest of the construction requires that \(\bVkp\) came from a perturbed three-term recurrence, which enters through \cref{cor:normalized-commutator}.
That corollary says that \(\bRkp\) and \(\bRk\) nearly intertwine \(\uTk\); we next show that the same holds for \(\Kkp\) and \(\Kk\), and for their powers.
Let \(\vec J\in\R^{(k+1)\times k}\) be the first \(k\) columns of \(\vec I_{k+1}\).
Since \(\Ckp\) and \(\Kkp\) are upper triangular, \(\Ckp\vec J=\vec J\Ck\) and \(\Kkp\vec J=\vec J\Kk\).

\begin{lemma}[Exchange rule]
\label{lem:exchange}
Let \(\Lk:=\Ck\bMk\Ck\).
Then \(\norm{\Kk}\le1\), \(\norm{\Ck}\le2\), and \(\norm{\Lk}\lesssim_k\eps\norm{\vec A}\).
Moreover, for every \(r\ge0\) there is \(\vec D_r\in\R^{k\times k}\) such that
\[
 \uTk(\Kk)^r-(\Kkp)^r\uTk=\vec J\vec D_r,
 \qquad
 \norm{\vec D_r}\le r\norm{\Lk}.
\]
\end{lemma}

\begin{proof}
Since \(\bGkp\succeq\vec0\), as it is a Gram matrix, \cref{eq:CGC-identity} gives \((\Kkp)^\T\Kkp\preceq\vec I\); hence \(\norm{\Kkp}\le1\), and by the triangle inequality \(\norm{\Ckp}=\norm{\vec I-\Kkp}\le2\).
The leading blocks inherit these bounds, so \(\norm{\Kk}\le1\) and \(\norm{\Ck}\le2\).

We can write \cref{eq:normalized-M-def} as \(\uTk\bRk-\bRkp\uTk=\vec J\bMk\).
Multiply on the left by \(\Ckp\) and on the right by \(\Ck\).
Since \(\bRk\Ck=\Kk\), \(\Ckp\bRkp=\Kkp\), and \(\Ckp\vec J=\vec J\Ck\), this gives
\[
 \Ckp\uTk\Kk-\Kkp\uTk\Ck=\vec J\Ck\bMk\Ck.
\]
Now expand \(\Ckp=\vec I-\Kkp\) and \(\Ck=\vec I-\Kk\) and cancel the terms \(\Kkp\uTk\Kk\) to obtain
\begin{equation}
 \uTk\Kk-\Kkp\uTk=\vec J\Lk.
 \label{eq:L-commutator}
\end{equation}
Up to notation, the first \(k\) rows of \cref{eq:L-commutator} are \cite[eq.~(3.9)]{paige2010}.
By \cref{cor:normalized-commutator}, \(\norm{\Lk}\le4\norm{\bMk}\lesssim_k\eps\norm{\vec A}\).

For the powers, induct on \(r\), starting from \(\vec D_0=\vec0\).
If the claim holds for \(r\), then by \cref{eq:L-commutator},
\begin{align*}
 &\uTk(\Kk)^{r+1}-(\Kkp)^{r+1}\uTk
 \\
 &\qquad=\bigl(\uTk(\Kk)^r-(\Kkp)^r\uTk\bigr)\Kk+(\Kkp)^r(\uTk\Kk-\Kkp\uTk)
 \\
 &\qquad=\vec J\bigl(\vec D_r\Kk+(\Kk)^r\Lk\bigr),
\end{align*}
so \(\vec D_{r+1}:=\vec D_r\Kk+(\Kk)^r\Lk\) works, and since \(\norm{\Kk}\le1\), \(\norm{\vec D_{r+1}}\le(r+1)\norm{\Lk}\).
\end{proof}

We can now show that the dilated vectors satisfy the computed recurrence up to a small residual.
To this end, define
\[
\hE:=\hA\hVk-\hVkp\uTk. 
\]

\begin{lemma}
\label{lem:dilated-recurrence}
It holds that $\norm{\hE}\lesssim_k\eps\norm{\vec A}$.
\end{lemma}

\begin{proof}
Since \(\hVk=\hVkp\vec J\), the \(r\)th blocks of \(\hVkp\) and \(\hVk\) are
\[
 \bVkp\Ckp(\Kkp)^r
 \qquad\text{and}\qquad
 \bVkp\Ckp(\Kkp)^r\vec J=\bVk\Ck(\Kk)^r.
\]
Hence, using \(\vec A\bVk=\bVkp\uTk+\bFk\) from \cref{eq:normalized-interface}, the \(r\)th block of the residual \(\hE\) is
\begin{align*}
 \vec E_r
 &:=\vec A\bVk\Ck(\Kk)^r-\bVkp\Ckp(\Kkp)^r\uTk
 \\
 &=\bFk\Ck(\Kk)^r+\bVkp\bigl(\uTk\Ck(\Kk)^r-\Ckp(\Kkp)^r\uTk\bigr).
\end{align*}
Expanding \(\Ckp=\vec I-\Kkp\) and \(\Ck=\vec I-\Kk\), the term in parentheses is \(\vec J\vec D_r-\vec J\vec D_{r+1}\) by \cref{lem:exchange}, and \(\bVkp\vec J=\bVk\), so
\[
 \vec E_r=\bFk\Ck(\Kk)^r+\bVk(\vec D_r-\vec D_{r+1}).
\]
With \(\norm{\bVk}\le\sqrt{k}\) and the bounds of \cref{lem:exchange}, this gives
\[
 \norm{\vec E_r}\le2\norm{\bFk}+(2k+1)\sqrt{k}\,\norm{\Lk},
 \qquad
 0\le r\le k.
\]
By \cref{lem:normalized-errors}, \(\norm{\bFk}\le\sqrt{k}\,\max_j\norm{\bar{\vec f}_j}\lesssim_k\eps\norm{\vec A}\), and \(\norm{\Lk}\lesssim_k\eps\norm{\vec A}\) by \cref{lem:exchange}.
Stacking the \(k+1\) blocks, \(\norm{\hE}\le\sqrt{k+1}\,\max_r\norm{\vec E_r}\lesssim_k\eps\norm{\vec A}\).
\end{proof}

\subsection{Symmetric correction and the backward theorem}
\label{sec:correction}

We now remove the residual \(\hE\) by a small symmetric perturbation of \(\hA\), which gives our backward model.

\begin{theorem}[Greenbaum-type backward model]
\label{thm:greenbaum}
Define
\[
 \tA:=\vec I_{k+1}\otimes\vec A + \bm{\Delta}, \quad \bm\Delta := -\hE\hVk^\T-\hVk\hE^\T+\hVk(\hVk^\T\hA\hVk-\Tk)\hVk^\T,
\]
and
\[
 \tv:=\col(\bar{\vec v}_1,\vec0,\ldots,\vec0)\in\R^N.
\]
Then $\tA$ is symmetric, $\|\bm\Delta\| \lesssim_k\eps\norm{\vec A}$, and 
\[
 \tA\hVk=\hVkp\uTk,
 \qquad
 (\hVkp)^\T\hVkp=\vec I, \quad \hVk\vec{e}_1 = \tv;
\]
in particular, \(k\) steps of exact Lanczos applied to \((\tA,\tv)\) produce exactly the computed \(\Tk\) and \(\beta_k\).
\end{theorem}

\begin{proof}
The matrix \(\tA\) is symmetric because \(\hA\) and \(\Tk\) are.
By \cref{lem:dilation}, \(\hVk^\T\hVkp=[\vec I,\vec0]\), so \(\hVk^\T\hVkp\uTk=\Tk\), and \cref{lem:dilated-recurrence} gives
\[
 \hVk^\T\hE=\hVk^\T\hA\hVk-\Tk.
\]
The right side is symmetric, so it also equals \(\hE^\T\hVk\).
Therefore, we can write 
\begin{equation}
  \label{eq:delta_as_projector}
   \bm\Delta=-\hE\hVk^\T-\hVk\hE^\T(\vec I-\hVk\hVk^\T).
\end{equation}
Using \cref{eq:delta_as_projector}, the fact that $\vec I-\hVk\hVk^\T$ is an orthogonal projector, and \cref{lem:dilated-recurrence}, we have that
\[
\norm{\bm\Delta}
\le2\norm{\hE}
\lesssim_k\eps\norm{\vec A}.
\]
Next, observe that  \cref{eq:delta_as_projector}, the fact that \(\hVk^\T\hVk=\vec I\), and \cref{lem:dilated-recurrence} imply that $\bm\Delta\hVk =-\hE$, and hence
\[
 \tA\hVk=\hA\hVk-\hE=\hVkp\uTk.
\]
By \cref{lem:dilation}, the columns \(\widehat{\vec v}_1,\ldots,\widehat{\vec v}_{k+1}\) of \(\hVkp\) are orthonormal, and \(\widehat{\vec v}_1=\tv\).
This gives the output of $k$ steps of Lanczos on $(\tA,\tv)$.
\end{proof}

\begin{remark}
\label{rem:wasserstein}
The backward model \((\tA,\tv)\) is close to \((\vec A,\bar{\vec v}_1)\) in the following sense.
First, by Weyl's inequality, every eigenvalue of \(\tA\) is within \(\norm{\bm\Delta}\lesssim_k\eps\norm{\vec A}\) of an eigenvalue of \(\vec A\).
Second, for every \(1\)-Lipschitz \(f:\R\to\R\), \(\abs{\bar{\vec v}_1^\T f(\vec A)\bar{\vec v}_1-\tv^\T f(\tA)\tv}\le \sqrt{2k}\,\norm{\bm\Delta} \lesssim_k\eps\norm{\vec A}\); that is, the spectral measures of \((\vec A,\bar{\vec v}_1)\) and \((\tA,\tv)\) are close in Wasserstein distance; see \cite{chen2024}.
Indeed, by construction \(\bar{\vec v}_1^\T f(\vec A)\bar{\vec v}_1=\tv^\T f(\hA)\tv\), so the difference is at most \(\norm{f(\hA)-f(\tA)}_{\mathsf F}\le\norm{\bm\Delta}_{\mathsf F}\le\sqrt{2k}\,\norm{\bm\Delta}\), where the first inequality holds because \(1\)-Lipschitz functions are \(1\)-Lipschitz in the Frobenius norm on symmetric matrices \cite{kittaneh1985}, and the second because \(\bm\Delta\) has rank at most \(2k\) by \cref{eq:delta_as_projector}.
\end{remark}

\begin{remark}
Our backwards matrix, which is nearly the direct sum of $\vec{A}$ with itself many times, is intuitively similar to the numerical models used in \cite{greenbaum_strakos1992,carson_liesen_strakos2024} to ``predict'' the behavior of finite precision Lanczos.  
\end{remark}

\subsection{Localization and containment}
\label{sec:ritz-consequences}

We now deduce two results of Paige on the location of the computed Ritz values from \cref{thm:greenbaum}.
Paige proved both by way of his loss-of-orthogonality analysis \cite[Lemma~3.1 and Theorem~3.1]{paige1980}, and Greenbaum's construction relies on that analysis and on the localization result \cite[\S4-5]{greenbaum1989}.
Our backward model, in contrast, uses only Paige's local error bounds (\cref{thm:local-errors}), allowing us to provide very simple proofs of these facts.

The first result states that a Ritz value whose scalar residual estimate \(\beta_k\abs{\vec e_k^\T\vec y}\) is small (in Paige's terminology, a Ritz value that has \emph{stabilized}) lies near the spectrum of \(\vec A\); this is \cite[Theorem~3.1]{paige1980}.
Paige's treatment of stabilized and clustered Ritz values was subsequently refined by W\"{u}lling \cite{wulling2005}; we do not pursue such refinements here.

\begin{theorem}[Localization of stabilized Ritz values]
\label{thm:stabilized}
Let \((\theta,\vec y)\) be a unit eigenpair of \(\Tk\).
Then there is an eigenvalue \(\lambda\) of \(\vec A\) such that
\[
 \abs{\theta-\lambda}
 \lesssim_k \beta_k\abs{\vec e_k^\T\vec y}+\eps\norm{\vec A}.
\]
\end{theorem}

\begin{proof}
Let \(\widehat{\vec x}:=\hVk\vec y\), which is a unit vector since \(\hVk\) has orthonormal columns.
By \cref{thm:greenbaum}, \(\tA\hVk=\hVkp\uTk\), and \(\uTk\vec y=\col(\theta\vec y,\beta_k\vec e_k^\T\vec y)\), so
\[
 (\tA-\theta\vec I)\widehat{\vec x}
 =\hVkp\uTk\vec y-\theta\hVk\vec y
 =\beta_k(\vec e_k^\T\vec y)\,\hvkp.
\]
For the symmetric matrix \(\hA\) and the unit vector \(\widehat{\vec x}\), expanding \(\widehat{\vec x}\) in an eigenbasis of \(\hA\) shows that some eigenvalue \(\lambda\) of \(\hA\), and hence of \(\vec A\), satisfies \(\abs{\theta-\lambda}\le\norm{(\hA-\theta\vec I)\widehat{\vec x}}\).
Hence, since \(\hvkp\) is a unit vector,
\[
 \abs{\theta-\lambda}
 \le\norm{(\tA-\theta\vec I)\widehat{\vec x}}
  +\norm{\bm\Delta}
  \le\beta_k\abs{\vec e_k^\T\vec y}+\norm{\bm\Delta}
 \lesssim_k\beta_k\abs{\vec e_k^\T\vec y}+\eps\norm{\vec A}.\qedhere
\]
\end{proof}

Next, we show that every computed Ritz value, stabilized or not, cannot be much smaller than the smallest eigenvalue of \(\vec A\) nor much larger than the largest eigenvalue of \(\vec A\).
This fact is essential to the analysis of Lanczos-based methods for matrix functions \cite{druskin_knizhnerman1991,knizhnerman1996,musco_musco_sidford2018}; see \cite{chen2024}.

\begin{theorem}[Containment of every computed Ritz value]
\label{thm:containment}
Every \(\theta\in\spec(\Tk)\) satisfies
\[
 \lambda_{\min}(\vec A)-\theta\lesssim_k\eps\norm{\vec A}
 \qquad\text{and}\qquad
 \theta-\lambda_{\max}(\vec A)\lesssim_k\eps\norm{\vec A}.
\]
\end{theorem}

\begin{remark}
This bound appears in \cite[Eq. 3.48]{paige1980}, but only a proof sketch appears.
This bound, in a form stronger than \cite[Eq. 3.48]{paige1980}, also appears in \cite[Theorem A.1]{paige2019}, and is proved via an entirely different argument.
\end{remark}

\begin{proof}
Let \((\theta,\vec y)\) be a unit eigenpair of \(\Tk\), and let \(\widehat{\vec x}:=\hVk\vec y\), which is a unit vector since \(\hVk\) has orthonormal columns.
Since \(\hVk^\T\tA\hVk=\hVk^\T\hVkp\uTk=\Tk\) by \cref{thm:greenbaum},
\[
 \theta
 =\vec y^\T\Tk\vec y
 =\widehat{\vec x}^\T\tA\widehat{\vec x}
 =\widehat{\vec x}^\T\hA\widehat{\vec x}
  +\widehat{\vec x}^\T\bm\Delta\widehat{\vec x}.
\]
The first term is a Rayleigh quotient of \(\hA\) at a unit vector, so it lies in \([\lambda_{\min}(\vec A),\lambda_{\max}(\vec A)]\), and the second has magnitude at most \(\norm{\bm\Delta}\lesssim_k\eps\norm{\vec A}\) by \cref{thm:greenbaum}.
\end{proof}

\section{A numerical illustration}
\label{sec:numerical}
\begin{figure}[t]
\centering
\includegraphics[scale=.86]{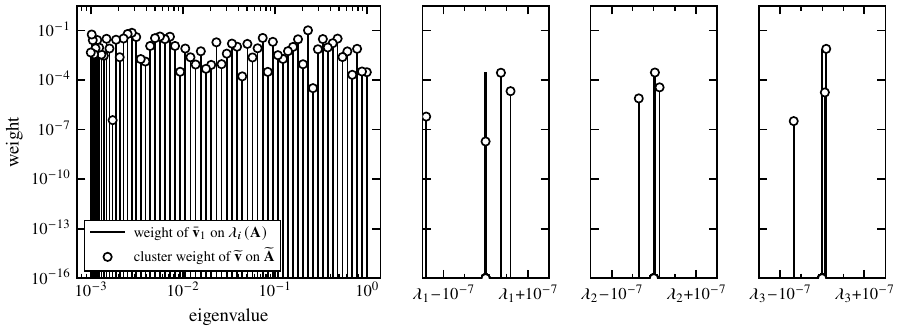}
\caption{Starting-vector weights of the backward model.
Left: the weight of \(\bar{\vec v}_1\) on each eigenvalue of \(\vec A\) (stems) and the total weight of \(\tv\) on the corresponding cluster of eigenvalues of \(\tA\) (circles); the two agree to within \(4\cdot10^{-6}\).
Right: zooms on the three largest eigenvalues, with the original eigenvalue and its weight repeated as a heavier stem: around each, the model spectrum contains a cluster of \(k+1=31\) eigenvalues within \(\norm{\bm\Delta}\approx2.5\cdot10^{-7}\) (tick marks at the bottom), over which the original weight is distributed; only a few cluster members carry weight above the plotted floor.}
\label{fig:spectral-measure}
\end{figure}

We illustrate the backward model of \cref{sec:backward} on the run shown in \cref{fig:backward-experiment}.
The matrix \(\vec A\) is the \(n=64\) Strako\v{s} matrix \cite{strakos1991}: diagonal, with eigenvalues (in decreasing order) \(\lambda_i=\lambda_n+\frac{n-i}{n-1}(\lambda_1-\lambda_n)\rho^{i-1}\), where \(\lambda_1=1\), \(\lambda_n=10^{-3}\), and \(\rho=0.9\).
\Cref{alg:lanczos} is run for \(k=30\) iterations in IEEE single precision (\(\eps\approx10^{-7}\)) from a random starting vector, and the backward model is then assembled in double precision, following the construction of \cref{sec:backward} verbatim.
The run loses orthogonality completely, \(\norm{\Vk^\T\Vk-\vec I}\approx1.1\).
Nevertheless, the dilated residual satisfies \(\norm{\hE}\approx2\cdot10^{-7}\,\norm{\vec A}\), and the perturbation \(\bm\Delta\) of \cref{thm:greenbaum} satisfies \(\norm{\bm\Delta}\approx2.5\cdot10^{-7}\).
Every eigenvalue of \(\tA\) therefore lies within \(\norm{\bm\Delta}\) of \(\spec(\vec A)\).

\Cref{fig:spectral-measure} compares the model \((\tA,\tv)\) with \((\vec A,\bar{\vec v}_1)\) through the weights of their starting vectors.
Each eigenvalue of \(\vec A\) is matched in \(\tA\) by a cluster of \(k+1\) eigenvalues within \(\norm{\bm\Delta}\), and the total weight of \(\tv\) on each cluster reproduces the weight of \(\bar{\vec v}_1\) on the corresponding eigenvalue of \(\vec A\), here to within \(4\cdot10^{-6}\).

\section*{AI Usage}

AI tools were used throughout the preparation of this note. 
The author assumes full responsibility for the content.

\printbibliography

\appendix

\section{Dependencies between results}
\label{sec:dependencies}

The diagram below shows how the results of \cref{sec:local,sec:loss,sec:backward} depend on one another.

\vspace{2em}
\centering
\begin{tikzpicture}[
  res/.style={draw, rounded corners=2pt, align=center, font=\small,
              inner sep=3pt, minimum width=34mm, minimum height=9mm},
  dep/.style={-{Stealth[length=2mm]}, semithick},
  x=1cm, y=1cm]
 % local analysis
 \node[res] (coef) at (2.3,0) {\cref{lem:coefficients}\\ A priori coefficient bounds};
 \node[res] (loc) at (2.3,-1.5) {\cref{thm:local-errors}\\ Local error bounds};
 \node[res] (glem) at (6.9,-1.5) {\cref{lem:gram-commutator}\\ Commutator for a recurrence};
 \node[res] (gcor) at (0,-3) {\cref{cor:gram-commutator}\\ Gram commutator};
 % loss of orthogonality
 \node[res] (paige) at (0,-4.5) {\cref{thm:paige}\\ Loss of orthogonality};
 % backward model
 \node[res] (nerr) at (4.6,-3) {\cref{lem:normalized-errors}\\ Normalized local errors};
 \node[res] (ncomm) at (6.9,-4.5) {\cref{cor:normalized-commutator}\\ Normalized commutator};
 \node[res] (dil) at (2.3,-6) {\cref{lem:dilation}\\ Orthonormal dilation};
 \node[res] (exch) at (6.9,-6) {\cref{lem:exchange}\\ Exchange rule};
 \node[res] (drec) at (6.9,-7.5) {\cref{lem:dilated-recurrence}\\ Dilated recurrence};
 \node[res] (green) at (6.9,-9) {\cref{thm:greenbaum}\\ Backward model};
 \node[res] (stab) at (5.0,-10.5) {\cref{thm:stabilized}\\ Localization};
 \node[res] (cont) at (8.8,-10.5) {\cref{thm:containment}\\ Containment};
 % edges
 \draw[dep] (coef) -- (loc);
 \draw[dep] (loc) -- (gcor);
 \draw[dep] (glem) -- (gcor);
 \draw[dep] (gcor) -- (paige);
 \draw[dep] (loc) -- (nerr);
 \draw[dep] (nerr) -- (ncomm);
 \draw[dep] (glem) -- (ncomm);
 \draw[dep] (ncomm) -- (exch);
 \draw[dep] (dil) -- (exch);
 \draw[dep] (exch) -- (drec);
 \draw[dep] (drec) -- (green);
 \draw[dep] (green) -- (stab);
 \draw[dep] (green) -- (cont);
\end{tikzpicture}

% The Typesetting note in style.tex describes the NBS presentation, not this
% one; it is left out until there is a note to match.
% \input{style}

\end{document}